\documentclass[12pt,letter,reqno]{amsart} 

\usepackage[text={420pt,660pt},centering]{geometry}
\usepackage{amssymb,amsfonts,amsthm,mathrsfs}
\usepackage{marginnote}
\usepackage{comment}
\usepackage{enumitem}
\usepackage{graphicx}
\usepackage{bm}

\usepackage[dvipsnames]{color}

\usepackage[colorlinks=true, pdfstartview=FitV, linkcolor=black, citecolor=blue, urlcolor=blue]{hyperref}

\def\Xint#1{\mathchoice
{\XXint\displaystyle\textstyle{#1}}%
{\XXint\textstyle\scriptstyle{#1}}%
{\XXint\scriptstyle\scriptscriptstyle{#1}}%
{\XXint\scriptscriptstyle%
\scriptscriptstyle{#1}}%
\!\int}
\def\XXint#1#2#3{{\setbox0=\hbox{$#1{#2#3}{%
\int}$ }
\vcenter{\hbox{$#2#3$ }}\kern-.6\wd0}}
\def\barint{\, \Xint -} 
\def\bariint{\barint_{} \kern-.4em \barint}
\def\bariiint{\bariint_{} \kern-.4em \barint}
\renewcommand{\iint}{\int_{}\kern-.34em \int} 
\renewcommand{\iiint}{\iint_{}\kern-.34em \int} 

\DeclareMathAlphabet{\mathcal}{OMS}{cmsy}{m}{n}

\theoremstyle{plain}

\newtheorem{theorem}{Theorem}[section]

\newtheorem{lemma}[theorem]{Lemma}

\theoremstyle{definition}
\newtheorem{remark}[theorem]{Remark}

\newcommand{\R}{\mathbb{R}}

\newcommand{\N}{\mathbb{N}}
\newcommand{\Z}{\mathbb{Z}}
\newcommand{\T}{\mathbb{T}}
\newcommand{\bP}{\mathbb{P}}

\newcommand{\supp}{\mathop{\mathrm{supp}}}

\newcommand{\p}{\partial}
\newcommand{\la}{\langle}
\newcommand{\ra}{\rangle}
\newcommand{\les}{\lesssim}
\newcommand{\ges}{\gtrsim}

\renewcommand{\:}{\colon}

\newcommand{\loc}{{\rm loc}}

\let\div\relax
\DeclareMathOperator{\div}{div}

\let\tilde\relax
\newcommand{\tilde}[1]{\widetilde{#1}}

\renewcommand{\L}{\bm{L}}
\newcommand{\B}{\bm{B}}
\newcommand{\G}{\bm{G}}
\renewcommand{\T}{\mathbb{T}}

\newcommand{\A}{\bm{A}}

\let\Re\relax
\DeclareMathOperator{\Re}{Re}

\numberwithin{equation}{section}
\setlist[enumerate]{leftmargin=*}

\title[]{Gluing non-unique Navier-Stokes solutions}
\author[Albritton]{Dallas Albritton} 
\address[Dallas Albritton]{School of Mathematics, Institute for Advanced Study, 1 Einstein Dr., Princeton, NJ 08540, USA}
\email{dalbritton@wisc.edu}

\author[Bru\'e]{Elia Bru\'e} 
\address[Elia Bru\'e]{School of Mathematics, Institute for Advanced Study, 1 Einstein Dr., Princeton, NJ 08540, USA}
\email{elia.brue@ias.edu}

\author[Colombo]{Maria Colombo}
\address[Maria Colombo]{Institute of Mathematics, EPFL SB, Station 8,  CH-1015 Lausanne, Switzerland }
\email{maria.colombo@epfl.ch}

\begin{document}
\begin{abstract}
We construct non-unique Leray solutions of the forced Navier-Stokes equations in bounded domains via gluing methods. This demonstrates a certain locality and robustness of the non-uniqueness discovered by the authors in~\cite{albritton2021non}.
\end{abstract}


\maketitle

\setcounter{tocdepth}{1}
\tableofcontents

\parskip   2pt plus 0.5pt minus 0.5pt

\section{Introduction}
\label{sec:intro}

In the recent work~\cite{albritton2021non}, we constructed non-unique Leray solutions of the Navier-Stokes equations in the whole space with forcing:
\begin{equation}
\label{eq:ns}
\tag{NS}
\begin{aligned}
    \p_t u + u \cdot \nabla u - \Delta u + \nabla p &= f \\
    \div u &= 0 \, .
\end{aligned}
\end{equation}
The non-unique solutions are driven by the extreme instability of a ``background" solution $\bar{u}$, which has a self-similar structure:
\begin{equation}
    \label{eq:background}
     \bar{u}(x,t)  = \frac{1}{\sqrt{t}} \bar{U} \left( \frac{x}{\sqrt{t}} \right) \, .
\end{equation}
In particular, the non-uniqueness ``emerges" from the irregularity at the space-time origin and is expected to be  local. However, while $\bar u$ is compactly supported, the non-uniqueness in~\cite{albritton2021non} 
involves another solution whose support is $\R^3 \times [0,T]$. Below, we demonstrate a certain locality and robustness of the non-uniqueness discovered in~\cite{albritton2021non} by gluing it into any smooth, bounded domain $\Omega \subset \R^3$ with no-slip boundary condition $u|_{\p \Omega} = 0$ and into the torus $\T^3 := \R^3 / (2\pi \Z)^3$, i.e., the fundamental domain  $[-\pi,\pi]^3$ with periodic boundary conditions.


\begin{theorem}[Non-uniqueness in bounded domains]
    \label{thm:introthm}
    Let $\Omega$ be a smooth, bounded domain in $\R^3$ or the torus $\T^3$. There exist $T>0$, $f \in L^1_t L^2_x(\Omega \times (0,T))$, and two distinct suitable Leray--Hopf solutions $u$, $\bar{u}$ to the Navier--Stokes equations
on $\Omega \times (0,T)$ with body force $f$, initial condition $u_0 \equiv 0$, and no-slip boundary condition. 
\end{theorem}

We assume a certain familiarity with the conventions of~\cite{albritton2021non}, although it will be convenient to recall the basics below. For $x\in \R^3$ and $t\in (0,+\infty)$, define the \emph{similarity variables}
\begin{equation}
    \xi = \frac{x}{\sqrt{t}} \, , \quad \tau = \log t \, .
\end{equation}
A velocity field $u$ and its \emph{similarity profile} $U$ are related via the transformation
\begin{equation}
    u(x,t) = \frac{1}{\sqrt{t}} U(\xi,\tau) \, .
\end{equation}
The pressure $p$, force $f$, and their respective profiles $P$, $F$ transform according to
\begin{equation}
    p(x,t) = \frac{1}{t} P(\xi,\tau) \, , \quad f(x,t) = \frac{1}{t^{3/2}} F(\xi,\tau) \, .
\end{equation} 
The Navier-Stokes equations in similarity variables are
\begin{equation}
    \label{eq:similaritynavierstokes}
\begin{aligned}
    \p_\tau U - \frac{1}{2} \left( 1 + \xi \cdot \nabla_\xi \right) U  - \Delta U + U \cdot \nabla U + \nabla P &= F \\
    \div U &= 0 \, .
    \end{aligned}
\end{equation}
Then $\bar{U} \in C^\infty_0(B_1)$ constructed in~\cite{albritton2021non} (see~\eqref{eq:background} above) is an unstable steady state of~\eqref{eq:similaritynavierstokes} with suitable smooth, compactly supported forcing term $\bar{F}$, and the non-unique solutions are trajectories on the unstable manifold associated to $\bar{U}$.

 In this paper, we take the following perspective. The force $f$ and one solution $\bar u$ are exactly the ones from~\cite{albritton2021non}. They are self-similar, smooth for positive times, and compactly supported inside the domain $\Omega$, which we assume contains the ball of radius $1/2$ centered at the origin. Each non-unique solution in~\cite{albritton2021non} constitutes then an ``inner solution" which lives at the self-similar scaling $|x| \sim t^{1/2}$, and this solution can be glued to an ``outer solution" (namely, $u \equiv 0$), which lives at the scaling $|x| \sim 1$. The boundary conditions are satisfied by the outer solution. 
 The solutions are glued by truncating on an intermediate scale $|x| \sim 1/10$. Let $\eta(x)$ be a suitable cut-off function with $\eta \equiv 1$ on $B_{1/9}$ and $\eta \equiv 0$ on $\R^3\setminus B_{1/7}$. Our main ansatz~is
 \begin{equation}
    \label{eq:ansatz}
 u = \bar{u} + \phi \eta + \psi \, ,
 \end{equation}
  where $\bar{u}$ is the compactly supported self-similar solution of the previous work, $\phi$ is the inner correction defined on the whole $\R^3$ (although only the values in $\supp \eta$ matter for the definition of $u$), and $\psi$ is the outer correction defined on the torus. Since $\phi$ is the inner correction, it will be natural to track its similarity profile $\Phi$ (we keep the lower and uppercase convention). We likewise decompose the pressure
  \begin{equation}
  p = \bar{p} + \pi \eta + q \, ,
  \end{equation}
  although $\bar{p} = 0$ from the construction in~\cite{albritton2021non}.

The PDE to be satisfied in $\Omega$ by 
$\phi$ and $\psi$ is
\begin{equation}
    \label{eq:iamsatisfied}
    \begin{split}   \partial_t(\phi\eta) &- \Delta (\phi\eta) +\bar{u} \cdot \nabla (\phi\eta) +\eta \phi\cdot \nabla \bar{u}+\eta \div(\eta \phi \otimes \phi ) +\bar{u} \cdot \nabla \psi+\psi \cdot \nabla \bar{u}\\
        &+ \partial_t\psi - \Delta \psi + \eta  \div(\phi\otimes \psi + \psi\otimes \phi) + \div(\psi\otimes \psi)  + \eta \phi (\phi \cdot \nabla \eta)
        \\& + \phi (\psi \cdot \nabla \eta) + \psi (\phi\cdot \nabla \eta)
        + \nabla (\pi \eta + q) = 0 \, ,
    \end{split}
\end{equation}
together with $\div (\phi \eta + \psi) = 0$.
We distribute the terms into an ``inner equation", which we think of as an equation for $\phi$ involving some terms in $\psi$, localized around the origin, and an ``outer equation", thought of as an equation for $\psi$. The inner and outer equations, when satisfied separately, imply that~\eqref{eq:iamsatisfied} is satisfied.


\subsection{Inner equation}
The inner equation has to be satisfied on the support of $\eta$, which is contained in $B_{1/7}$:
\begin{equation}\label{eq:inner eq}
\begin{split}
\partial_t \phi - \Delta \phi& + \bar{u} \cdot \nabla \phi + \phi \cdot \nabla \bar{u} + \div(\eta \phi \otimes \phi)
\\& \quad +  \div(\psi \otimes \phi + \phi \otimes \psi) + \bar{u} \cdot \nabla \psi + \psi \cdot \nabla \bar{u} + \nabla \pi  =0\, ,
\end{split}
\end{equation}
and it is coupled to the divergence-free condition 
\begin{equation}
    \div \phi = 0\, .
\end{equation}
We introduce the operator $\L_{\rm ss}$, i.e., the linearized operator of~\eqref{eq:similaritynavierstokes} around $\bar{U}$:
\begin{equation}
\label{eq:linearizedsimilaritynavierstokes}
    -\L_{\rm ss} \Phi = - \frac{1}{2} \left( 1 + \xi \cdot \nabla_\xi \right) \Phi  - \Delta \Phi + \bP \left( \bar{U} \cdot \nabla \Phi + \Phi \cdot \nabla \bar{U} \right) \, .
\end{equation}
In self-similar variables, we rewrite the cut-off $\eta(x) = N(\xi,\tau)$. We rewrite the inner equation~\eqref{eq:inner eq} as
\begin{equation}
    \label{eqn:inner}
    \begin{split}
    \partial_\tau \Phi -\L_{\rm ss} &\Phi + \Phi \cdot \nabla (\Phi N) 
    + 
    \div ( \tilde{N} \Psi \otimes \Phi + \tilde{N} \Phi \otimes \Psi)
       \\ & \quad 
    + \bar{U} \cdot \nabla \Psi + \Psi \cdot \nabla \bar{U}+ \nabla \Pi =0\, ,
    \end{split}
\end{equation}
where $\tilde{N}(\xi,\tau) = N(\xi/3,\tau)$. We now require that it is satisfied in the whole $\R^3$, not merely on the support of $N$.

\subsection{Outer equation}
Using that $(\bar{u} \cdot \nabla \eta) \phi = 0$ and $\partial_t \eta = 0$, as a consequence of our choice of $\eta$, we deduce the following system for the outer equation:
\begin{equation}
    \begin{cases}
    \label{eqn:outer}
    &\partial_t \psi  - \Delta \psi +  \div(\psi \otimes \psi) + (\psi \cdot \nabla \eta) \phi + (\phi \cdot \nabla \eta) \psi 
    \\  & \qquad
     - \phi \Delta \eta - 2 \nabla \phi \cdot \nabla  \eta  + \eta \phi(\phi \cdot \nabla \eta) + \pi \nabla \eta + \nabla q  =0 
    \\&
    \div \psi = - \nabla \eta \cdot \phi 
    \end{cases}
\end{equation}
The problem~\eqref{eqn:outer} is to be solved in $\Omega$ with the boundary condition $\psi|_{\p \Omega} = 0$.

\medskip

We now consider the PDEs~\eqref{eq:inner eq} and~\eqref{eqn:outer} as a system for $(\Phi,\psi)$. The two components will be controlled using two different linear operators, $\L_{\rm ss}$ and $\bP \Delta$.

In dividing the terms of~\eqref{eq:iamsatisfied} into the inner and outer equations, we put the ``boundary terms", i.e., terms involving derivatives of $\eta$, into the outer equation, whereas the we put the terms $\bar{U} \cdot \nabla \Psi$ and $\Psi \cdot \nabla \bar{U}$ into the inner equation.

Crucially, we expect that \emph{the boundary terms are small because solutions of the inner equation are well localized.} Consequently, \emph{$\psi$ decouples from $\phi$ as $t \to 0^+$}, and therefore the linear part of the system should be invertible.\footnote{One can compare this to the matrix
$\begin{bmatrix}
a & b \\
\varepsilon & d
\end{bmatrix}$ where $\varepsilon$ represents the boundary terms, $b$ represents the $\bar{U} \cdot \nabla \Psi + \Psi \cdot \nabla \bar{U}$ terms, and the diagonal elements $a$ and $d$ are $O(1)$. In fact, eventually we will see that $\psi$ decays faster than $\phi$ as $t \to 0^+$, so the terms corresponding to $b$ are small, and the whole system decouples.} For this to work, it is necessary to show that the boundary terms are negligible, which requires knowledge of the inner correction $\Phi$ in \emph{weighted spaces}.


With this knowledge, we solve the full nonlinear system via a fixed point argument. The details of the scheme will be discussed in Section~\ref{sec:integraleqns}.

\medskip

Our method is inspired by the parabolic ``inner-outer" gluing technique exploited in~\cite{davila2020singularity} to analyze bubbling and reverse bubbling in the two-dimensional harmonic map heat flow into $\mathbb{S}^2$. The reverse bubbling in~\cite{davila2020singularity} is also an example of gluing techniques applied to non-uniqueness, although its mechanism is quite different. It is worth noting that, in that setting, the harmonic map heat flow actually has a natural uniqueness class~\cite{Struwe1985}. 






We expect that Theorem~\ref{thm:introthm} may be extended in a number of ways. Our techniques extend with minimal effort to non-uniqueness centered at $k$ points. 
We expect that the conditionally non-unique solutions of Jia and \v{S}ver{\'a}k~\cite{jiasverakillposed} can also be glued.\footnote{For this, it may be necessary to assume that the self-similar solution is just barely unstable, as is done in the truncation procedure in~\cite{jiasverakillposed}. Typically, the background solution $\bar{u}$ must be cut in the gluing procedure, but we avoid this because in our setting $\bar{u}$ is already compactly supported.}  Finally, it would be interesting to glue the two-dimensional Euler constructions of~\cite{Vishik1,Vishik2} (see also~\cite{OurLectureNotes}) into the torus or bounded domains. This is likely to be more challenging than the present work, since the Euler equations are quasilinear and the construction of the unstable manifold more involved. We leave these and other extensions to future work.


\section{Preliminaries} 

Consider $p \in (1,+\infty)$ and $\Omega = \R^3, \T^3$, or a smooth, bounded domain in $\R^3$.

We define
\begin{equation}
L^p_\sigma(\Omega) := \overline{\{ \phi \in C^\infty_c(\Omega; \mathbb{R}^3)\, : \, \div \phi = 0\}}^{ L^p(\Omega;\R^3)} \, ,
\end{equation}
which
can be understood as the space of $L^p$ velocity fields with $\div \phi = 0$ on $\Omega$ and $\phi \cdot \nu = 0$ on $\partial \Omega$, where $\nu$ is the exterior normal to $\Omega$. See~\cite[Chapter III]{galdi} or~\cite[Lemma 1.4]{Tsaibook}. 
Notice that the boundary condition is vacuous when $\Omega = \R^3,\T^3$. 

There exists a bounded projection $\mathbb{P}: L^p(\Omega;\R^3) \to L^p_\sigma(\Omega)$ satisfying $\mathbb{P}\phi = \phi - \nabla \Delta^{-1}_N\div \phi$ for any $\phi \in C^\infty_c(\Omega; \mathbb{R}^3)$, where $\Delta_N$ is the Neumann Laplacian. This is the \emph{Leray projection}. By density of divergence-free test fields, it agrees across $L^p$ spaces and, in particular, with the extension of the $L^2$-orthogonal projection onto divergence-free fields; see~\cite[Chapter III]{galdi} or~\cite[Theorem 1.5]{Tsaibook}.  

\subsection{Linear instability} 

The following theorem provides an unstable background for the $3$D Navier-Stokes equations. We refer the reader to \cite{albritton2021non} for its proof.

\begin{theorem}[Linear instability]\label{thm:linear insta}
There exists a divergence-free vector field $\bar{U}\in C^\infty(\R^3;\R^3)$ with $\supp \bar{U} \subset B_1(0)$ such that the linearized operator $\L_{\rm ss} \: D(\L_{\rm ss}) \subset L^2_\sigma(\R^3) \to L^2_\sigma(\R^3)$ defined by
\begin{equation}
     - \L_{\rm ss} U = - \frac{1}{2} \left( 1 + \xi \cdot \nabla_\xi \right) U - \Delta U + \bP (\bar{U} \cdot \nabla U + U \cdot \nabla \bar{U}) \, ,
\end{equation}
where $D(\L_{\rm ss}) := \{ U \in L^2_\sigma : U \in H^2(\R^3), \,  \xi \cdot \nabla U \in L^2(\R^3) \}$,
has a maximally unstable eigenvalue~$\lambda$ with non-trivial smooth eigenfunction $\rho$ belonging to $H^k(\R^3)$ for all $k \geq 0$: 
\begin{equation}
    \L_{\rm ss} \rho = \lambda \rho \quad \text{ and } \quad a := \Re \lambda  = \sup_{z \in \sigma(\L_{\rm ss})} \Re z > 0 \, .
\end{equation}
\end{theorem}

The construction in~\cite{albritton2021non} allows $\bar{U}$ to be chosen to make $a$ arbitrarily large, and it will be convenient, though not strictly necessary, to enforce that $a \geq 10$.

We can now define
\begin{equation}\label{eq: U lin}
U^{\rm lin}(\cdot,\tau) = \Re (e^{\lambda \tau} \rho) \, ,
\end{equation}
a solution of the linearized PDE $\p_\tau U^{\rm lin} = \L_{\rm ss} U^{\rm lin}$,
 with maximal growth rate $a \geq 10$.

The following lemma, borrowed from \cite[Lemma 4.4]{albritton2021non}, provides sharp growth estimates on the semigroup $e^{\tau \L_{\rm ss}}$.

 \begin{lemma}\label{prop:semigroup}
	Let $\bar{U}$ be as in Theorem \ref{thm:linear insta}. Then, for any $\sigma_2 \geq \sigma_1 \geq 0$ and $\delta > 0$, it holds
	\begin{equation}\label{eq:reg}
		\| e^{ \tau \L_{\rm ss}} U \|_{H^{\sigma_2}} \les_{\sigma_1,\sigma_2,\delta} \tau^{-\frac{(\sigma_2-\sigma_1)}{2}} 
		e^{\tau (a + \delta)} \| U \|_{H^{\sigma_1}} \, ,
	\end{equation}
	for any $U \in L^2_\sigma \cap H^{\sigma_1}(\R^3)$.
\end{lemma}

\subsection{Improved space decay}

For $\zeta \in \R$ and $p \in [1,+\infty]$, define $L^p_\zeta(\R^3)$ to be the space of $f \in L^p_\loc(\R^3)$ satisfying
\begin{equation}
    \| f \|_{L^p_\zeta} := \| \la \cdot \ra^{\zeta} f \|_{L^p} < + \infty \, ,
\end{equation}
where $\la \xi \ra = (1+|\xi|^2)^{1/2}$ is the Japanese bracket notation.
We further define
\begin{equation}
    L^p_w(\R^3) := L^p_{4}(\R^3) \, .
\end{equation}

\begin{lemma}
\label{lemma:spacialdecay}
Let $\zeta \in (3,4]$, $p \in (3,+\infty]$ and $\delta>0$. Then
\begin{equation}
    \| e^{\tau \L_{\rm ss}} \bP \div \|_{L^p_\zeta \to L^\infty_\zeta} \les_{\delta,\zeta,p}  \tau^{-(\frac{1}{2}+\frac{3}{2p})} e^{(a+\delta) \tau} \, .
\end{equation}
\end{lemma}

\begin{remark}
    \label{rmk:wellposednessrmk}
For $M \in L^p(\R^3;\R^{3 \times 3})$ and $p \in [1,+\infty]$, the solution operator $e^{\tau \L_{\rm ss}} \bP \div M$ is easily shown to be well defined by standard arguments. Namely, consider the solution $u$ to the following PDE:
\begin{equation}
    \p_t u - \Delta u + \bP \div (\bar{u} \otimes u + u \otimes \bar{u}) = 0 \, , \quad u(\cdot,1) = \bP \div M \, .
\end{equation}
The mild solution theory of the above PDE can be developed using properties of the semigroup $e^{t \Delta} \bP \div$ (whose kernel consists of derivatives of the Oseen kernel, see~\eqref{eq:Oseen1}-\eqref{eq:Oseen2} below) by considering $\bP \div (\bar{u} \otimes u + u \otimes \bar{u})$ as a perturbation in Duhamel's formula. In particular, it is standard to demonstrate that, for all $T > 1$ and $t \in (1,T]$, we have
\begin{equation}
    \label{eq:thingweknow}
    \| u(\cdot,t) \|_{L^q} \les_{T,p,q} (t-1)^{-[\frac{1}{2} + \frac{3}{2}(\frac{1}{p}-\frac{1}{q})]} \| M \|_{L^p} \, ,
\end{equation}
for all $1 \leq p \leq q \leq +\infty$. Finally, we define $e^{\tau \L_{\rm ss}} \bP \div M := U \: \R^3 \times [0,+\infty) \to \R^3$ according to
\begin{equation}
    u(x,t) = \frac{1}{\sqrt{t}} U(\xi,\tau) \, .
\end{equation}
With this in mind, we focus below on growth estimates for the semigroup.
\end{remark}

\begin{proof}[Proof of Lemma~\ref{lemma:spacialdecay}]
To begin, we establish weighted estimates for the semigroup $e^{\tau \A} \bP \div$, where
\begin{equation}
    - \A := - \frac{1}{2} \left( 1 + \xi \cdot \nabla \right) - \Delta \, .
\end{equation}
For $M \in L^p_\zeta(\R^3;\R^{3\times3}) \subset L^2$, consider the solution $u : \R^3\times [1,+\infty) \to \R^3$ to
\begin{equation}
    \p_t u - \Delta u = 0 \, , \quad u(\cdot,1) = \bP \div M \, .
\end{equation}
We have the representation formula
\begin{equation}
    \label{eq:Oseen1}
    u(x,t) = g(\cdot,t-1) \ast M \, ,
\end{equation}
where $g$ is tensor-valued and consists of derivatives of the Oseen kernel (see, e.g.,~\cite[p. 80]{Tsaibook}),
\begin{equation}
    \label{eq:Oseen2}
    g = \frac{1}{t^2} G\left( \frac{x}{\sqrt{t}} \right) \, ,
\end{equation} satisfying the pointwise estimate
\begin{equation}
    |G(\xi)| \les \la \xi \ra^{-4} \, .
\end{equation}
Define $e^{\cdot \A} \bP \div M := U \: \R^3 \times [0,+\infty) \to \R^3$ according to
\begin{equation}
    u(x,t) = \frac{1}{\sqrt{t}} U(\xi,\tau) \, .
\end{equation}
Using the representation formula and elementary estimates for convolution (see Lemma~\ref{lem:convolutioninequality} and Remark~\ref{rmk:convolutionremark}), we have two estimates. First, we have the short-time estimate
\begin{equation}
    \label{eq:shorttimeest}
    \| u \|_{L^\infty_\zeta} \les_{\zeta,p} (t-1)^{-(\frac{1}{2}+\frac{3}{2p})}  \| M \|_{L^p_\zeta} \, , \quad t \in (1,e] \, ,
\end{equation}
which implies that
\begin{equation}
    \label{eq:Aest}
 \| U \|_{L^\infty_\zeta} \les_{\zeta,p} \tau^{-(\frac{1}{2}+\frac{3}{2p})} \| M \|_{L^p_\zeta} \, , \quad \tau \in (0,1] \, .
\end{equation}
Moreover, we have the long-time estimate
\begin{equation}
    \label{eq:Aestlong}
    \| U \|_{L^\infty_\zeta} \les_{\zeta,p} \| M \|_{L^p_\zeta} \, , \quad \tau \in [1,+\infty) \, .
\end{equation}
This completes the semigroup estimates for $e^{\tau \A} \bP \div$.

We now turn our attention to the growth estimate for $e^{\tau \L_{\rm ss}} \bP \div$. First, we prove
\begin{equation}
    \label{eq:alreadyknow}
    \| e^{\tau \L_{\rm ss}} \bP \div M \|_{L^\infty} \les_{\delta,p} \tau^{-(\frac{1}{2}+\frac{3}{2p})} e^{\tau(a+\delta)} \| M \|_{L^p} \, , \quad \tau > 0 \, .
\end{equation}
We already have this estimate for $\tau \in (0,2]$, see~\eqref{eq:thingweknow} in Remark~\ref{rmk:wellposednessrmk}, so we focus on $\tau \geq 2$.
This is done by splitting $e^{\tau \L_{\rm ss}} \bP \div = e^{(\tau-1) \L_{\rm ss}} \bP \circ e^{\L_{\rm ss}} \bP \div$, using estimate~\eqref{eq:thingweknow} (with $p=q=2$) for the operator $e^{\L_{ss}} \bP \div$, and using the growth estimate
\begin{equation}
    \| e^{\tau \L_{\rm ss}} \bP \|_{L^2 \to H^2} \les_\delta \tau^{-1} e^{\tau (a+\delta)} \, , \quad \tau > 0 \, ,
\end{equation}
from Lemma~\ref{prop:semigroup}, for the operator $e^{(\tau-1) \L_{\rm ss}} \bP$, along with Sobolev embedding $H^2 \subset L^\infty$ in dimension three. With~\eqref{eq:alreadyknow} in hand, we proceed with the desired $L^\infty_\zeta$ estimate. Define $U := e^{\tau \L_{\rm ss}} \bP \div M$ and write
\begin{equation}
    U(\cdot,\tau) = e^{\tau \A} \bP \div M - \int_0^\tau e^{(\tau-s) \A} \bP \div (\bar{U} \otimes U + U \otimes \bar{U} ) \, ds \, .
\end{equation}
We will combine the semigroup estimates~\eqref{eq:Aest} and~\eqref{eq:Aestlong} for $\A$ with~\eqref{eq:alreadyknow} and the fact that $\bar{U}$ is compactly supported. We end up with
\begin{equation}
\begin{aligned}
        \| U \|_{L^\infty_\zeta} &\les_{\delta,p} \max(\tau^{-(\frac{1}{2}+\frac{3}{2p})},1) \| M \|_{L^p_\zeta} 
        \\& \quad 
        +
        \int_0^\tau \max((\tau - s)^{-\frac{1}{2}},1) \| (\bar U \otimes U + U \otimes \bar U)(\cdot, s) \|_{L^\infty_\zeta }\, ds
        \\
        &\les_{\delta,p} \max(\tau^{-(\frac{1}{2}+\frac{3}{2p})},1) \| M \|_{L^p_\zeta} 
        +
        \int_0^\tau \max((\tau - s)^{-\frac{1}{2}},1) \| U (\cdot, s) \|_{L^\infty }\, ds
        \\
        &
        \les_{\delta,p} \max(\tau^{-(\frac{1}{2}+\frac{3}{2p})},1) \| M \|_{L^p_\zeta}
        + \int_0^\tau \max((\tau - s)^{-\frac{1}{2}},1) s^{-(\frac{1}{2}+\frac{3}{2p})}  e^{s(a+\delta)} \| M \|_{L^p} \, ds \\
        &\les_{\delta,p} \max(\tau^{-(\frac{1}{2}+\frac{3}{2p})},1) e^{\tau(a+\delta)} \, ds \, ,
\end{aligned}
\end{equation}
where we used that $p>3$.
This holds for all $\delta > 0$, completing the proof.
\end{proof}


\begin{lemma}
\label{lem:efndecay}
The eigenfunction $\rho$ in Theorem~\ref{thm:linear insta} belongs to $L^\infty_{w}(\R^3)$.
\end{lemma}

\begin{proof}
The proof is akin to~\cite[Corollary 3.3]{albritton2021non}: $\rho \in D(\L_{\rm ss})$ solves
\begin{equation}\label{eqn:omega}
   \lambda \rho - \frac{1}{2}(1+\xi \cdot \nabla_\xi) \rho - \Delta \rho = \bP \div F
\end{equation}
where $- F = \bar{U} \otimes \rho + \rho \otimes \bar{U}$. Notably, local elliptic regularity implies that $\rho$ is smooth on the support of $\bar{U}$. Hence, $F \in L^\infty_w$. Next, we `undo' the similarity variables by defining
\begin{equation}
    h(x,t) = t^{\lambda-\frac{1}{2}} \rho \left(\frac{x}{\sqrt{t}}\right) \, , \quad 
    M(x,t) = t^{\lambda-1} F\left(\frac{x}{\sqrt{t}} \right) \, .
\end{equation}
Then
\begin{equation}
    \label{eq:definitionofh}
    \p_t h - \Delta h = \bP \div M \, , \quad h(\cdot,0) = 0 \, ,
\end{equation}
and we have the representation formula
\begin{equation}
    \rho = h(\cdot,1) = \int_0^1 e^{\Delta(1-s)} \bP \div M(\cdot,s) \, ds \, ,
\end{equation}
which yields (see~\eqref{eq:shorttimeest})
\begin{equation}
    \| \rho \|_{L^\infty_w} \les \int_0^1 (1-s)^{-\frac{1}{2}} \| M(\cdot,s) \|_{L^\infty_w} \, ds \les \int_0^1 (1-s)^{-\frac{1}{2}} s^{\Re \lambda - 1} \, ds \| F \|_{L^\infty_w} < +\infty
\end{equation}
since $\Re \lambda > 0$. Here, we used that $\| f(x/\ell) \|_{L^\infty_w} \leq \|f\|_{L^\infty_w}$ for $\ell \in (0,1]$. This completes the proof.
\end{proof}

\subsection{Stokes equations in bounded domains}
We now turn our attention to the linear theory for the outer equation. We begin with semigroup theory for the Stokes equations, see \cite[Sections 2 and 5]{Hieber2016} and~\cite[Chapter 5]{Tsaibook}.



\begin{lemma}[Stokes in bounded domains]
\label{lem:stokessemigroup}
Let $p \in (1,+\infty)$ and $\Omega \subset \R^3$ be a smooth, bounded domain. Define
\begin{equation}
    D(A) := W^{2,p} \cap W^{1,p}_0 \cap L^p_\sigma(\Omega)
\end{equation}
and the Stokes operator
\begin{equation}
    A = \bP \Delta : D(A) \to L^p_\sigma(\Omega) \, .
\end{equation}  Then the Stokes operator $A$ generates an analytic semigroup $(e^{t A})_{t \geq 0}$, and we have, for all $p \in (1,+\infty)$ and $q \in [p,+\infty]$, the smoothing estimates
\begin{equation}
    \| e^{t A} \bP \|_{L^p \to L^q} + t^{\frac{1}{2}} \| e^{t A} \bP \div \|_{L^p \to L^q} \les t^{\frac{3}{2}(\frac{1}{q}-\frac{1}{p})} \, .
\end{equation}
\end{lemma}

The function $u(x,t) = (e^{tA} u_0)(x)$ solves the Stokes equations with no-slip boundary conditions
\begin{equation}
    \partial_t u - \Delta u + \nabla \pi  = 0 \, ,
    \qquad
    u(\cdot,t) = 0 \, \, \, \text{on $\partial \Omega$}\, ,
\end{equation}
for any $u_0 \in L^p_\sigma(\Omega)$. The boundary conditions are built into the domain of the operator, and $e^{t A}: L^p_\sigma \to D(A)$ for any $t>0$.

\medskip

To solve the Stokes equations with non-zero divergence, we use the following lemma due to~\cite[Theorem 4]{farwig2006new}.
\begin{lemma}[Stokes with inhomogeneous divergence]
    \label{lem:stokesdiv}
Let $T > 0$ and $\Omega \subset \R^3$ be a smooth, bounded domain.
For $p \in (3,+\infty)$, and $r \in (1,+\infty)$, consider $h \in L^r_t L^p_x(\Omega \times (0,T))$ with zero mean: $\int_\Omega h(x,t) \, dx = 0$ for a.e. $t \in (0,T)$.

Then there exists a unique very weak solution $u \in L^r_t L^p_x(\Omega \times (0,T))$ to the following Stokes problem in $\Omega \times (0,T)$:
\begin{equation}
\left\lbrace
\begin{aligned}
    \p_t u - \Delta u + \nabla \pi &= 0 \\
    \div u &= h \\
    u|_{\p \Omega \times (0,T)} &= 0 \\
    u(\cdot,0) &= 0 \, ;
    \end{aligned}
    \right.
\end{equation}
that is, for all divergence-free $w \in C^1_c([0,T);(C^2 \cap C_0)(\bar{\Omega}))$, we have
\begin{equation}\label{eq:weak formulation}
    \int_0^T \int_\Omega u (-\p_t - \Delta) w \, dx \, dt = 0
\end{equation}
and $\div u = h$ in the sense of distributions on $\Omega \times (0,T)$.
Moreover, $u$ satisfies the estimate
\begin{equation}
    \| u \|_{L^r_t L^p_x(\Omega \times (0,T))} \les_{\Omega,r,p} \| h \|_{L^r_t L^p_x (\Omega \times (0,T))} \, .
\end{equation}
\end{lemma}

\begin{remark}
    \label{rmk:veryweaksoltuionremark}
    The initial condition $u(\cdot, 0) = 0$ is understood ``modulo gradients''.
    Moreover, it can be proven (cf. ~\cite[Theorem 4, Remark 3]{farwig2006new}) that $A^{-1}\mathbb{P}u\in C([0,T);L^p_\sigma(\Omega))$ and $A^{-1}\mathbb{P}u(\cdot,0)=0$. Notably, \emph{uniqueness holds} in the above class of very weak solutions, which makes the notion a useful generalization. 
\end{remark}



\subsection{Stokes equations in the periodic domain}\label{subsec:stokes periodic}



On the torus $\T^3:= \R^3/(2\pi \mathbb{Z})^3$, the Stokes equations can be solved by means of the heat semigroup, since the Stokes operator $A$ in $L^p_\sigma(\T^3)$, $p\in (1,+\infty)$, coincides with
\begin{equation}\label{eq:Delta}
    \Delta : W^{2,p}\cap L^p_\sigma(\T^3) \to L^p_\sigma(\T^3) \, .
\end{equation}
Hence, the associated Stokes semigroup $(e^{t A})_{t\ge 0}$ coincides with the heat semigroup and enjoys the smoothing estimates
\begin{equation}\label{eq:Stokes est torus}
    \| e^{t A} \bP \|_{L^p \to L^q} + t^{\frac{1}{2}} \| e^{t A } \bP \div \|_{L^p \to L^q} \les t^{\frac{3}{2}(\frac{1}{q}-\frac{1}{p})} \, ,
\end{equation}
for all $p \in (1,+\infty)$ and $q \in [p,+\infty]$.

\medskip

The Stokes equations with non-zero divergence,
\begin{equation}\label{eq:Stokes div}
\left\lbrace
\begin{aligned}
    \p_t u - \Delta u + \nabla \pi &= 0 \\
    \div u &= h \\
    u(\cdot,0) &= 0 \, ,
    \end{aligned}
    \right.
\end{equation}
admit an explicit solution
\begin{equation}
    \label{eq:theformulaiwant}
    u = \nabla \Delta^{-1} h \, ,
\end{equation}
provided $h$ satisfies the compatibility condition $\int_{\T^3} h(x,t)\, dx = 0$ for a.e. $t\in (0,T)$. The solution is in the very weak sense, that is, $\div u = h$ in the sense of distributions, and, for all $w \in C^1_c([0,T);C^2(\T^3))$, we have
\begin{equation}
\label{eq:weakform}
  \int_0^T \int_{\T^3} u(-\partial_t w - \Delta w) \, dx \, dt = 0 \, .
\end{equation}
As in Remark~\ref{rmk:veryweaksoltuionremark}, the initial condition is only ``modulo gradients".



It is immediate to check that
\begin{equation}\label{eq:psidiv torus}
    \| u \|_{L^r_t L^p_x(\T^3 \times (0,T))} \les_{p} \| h \|_{L^r_t L^p_x (\T^3 \times (0,T))} \, ,
\end{equation}
for any $r\in [1,\infty]$ and $p\in (1,\infty)$.

Moreover, there is uniqueness when $u \in L^r_t L^p_x(\T^3 \times (0,T))$. That is, necessarily $u$ is given by~\eqref{eq:theformulaiwant}. Indeed, if $\div u = 0$, then $u = \bP u$, and~\eqref{eq:weakform} simply asserts that $u$ solves the heat equation with zero initial condition.


\subsection{Weighted pressure estimates}

To estimate the boundary term $\pi \nabla \eta$ in~\eqref{eqn:outer}, where $\pi$ is the ``inner pressure", we require estimates for the singular integral operator $(-\Delta)^{-1} \div \div$ in weighted spaces. Notably, $(-\Delta)^{-1} \div \div F =  \bm{R} \otimes \bm{R} : F$, where $F : \R^3 \to \R^{3\times 3}$ is a tensor and $\bm{R} = (R_1,R_2,R_3)$ is the vector of Riesz transforms $R_i$, whose kernels are $c_3 \xi_i/|\xi|^{4}$.

For $F \in L^1(\R^3)$ compactly supported in $B_R$ with $R > 0$, we evidently have
\begin{equation}
	\label{eq:compactsupportest}
| (-\Delta)^{-1} \div \div F | \lesssim \langle \xi \rangle^{-3} \| F \|_{L^1(B_R)} \, , \quad |\xi| \geq 2R \, .
\end{equation}

For $F \in L^p_w(\R^3)$ with $p \in (1,+\infty)$ and $R \geq 2$, we require the estimate
\begin{equation}
	\label{eq:weightedlppressureest}
\| (-\Delta)^{-1} \div \div F \|_{L^p(B_{10 R} \setminus B_R)} \lesssim_p R^{-3+\frac{3}{p}} \| F \|_{L^p_w} \, .
\end{equation}
We split $F = F \mathbf{1}_{B_{20 R} \setminus B_{R/2}} + F (1 - \mathbf{1}_{B_{20 R} \setminus B_{R/2}})$. Then, in the near field, we have
\begin{equation}
	\label{eq:nearfield}
\| (-\Delta)^{-1} \div \div F\mathbf{1}_{B_{20 R} \setminus B_{R/2}} \|_{L^p(B_{10 R} \setminus B_R)} \lesssim_p R^{-4} \| F \|_{L^p_w} \, ,
\end{equation}
whereas, whenever $\xi \in B_{10 R} \setminus B_R$, we have the contribution
\begin{equation}
| (-\Delta)^{-1} \div \div F (1 - \mathbf{1}_{B_{20 R} \setminus B_{R/2}}) | \lesssim (\langle \cdot \rangle^{-3} \ast |F|)(\xi) \lesssim_p \langle \xi \rangle^{-3} \| F \|_{L^p_w} \, ,
\end{equation}
as in Remark~\ref{rmk:convolutionremark}, from the far field. Hence,
\begin{equation}
	\label{eq:farfield}
\| (-\Delta)^{-1} \div \div F (1 - \mathbf{1}_{B_{20 R} \setminus B_{R/2}}) \|_{L^p(B_{10 R} \setminus B_R)} \lesssim_p R^{-3+\frac{3}{p}} \| F \|_{L^p_w} \, ,
\end{equation}
and the estimate follows by combining~\eqref{eq:nearfield} and~\eqref{eq:farfield}. In practice, this estimate will sometimes be coupled with the embedding $L^\infty_{8}(\R^3) \subset L^p_w(\R^3)$.

\section{The integral equations}
\label{sec:integraleqns}


In what follows $\Omega$ is either a smooth, bounded domain or the periodic box $\T^3$.
For $\bar \tau \in \R$, $\bar t>0$, and $\alpha, \beta>0$, we define the norms
\begin{equation}\label{eq:X}
    \| \Phi \|_{X^\alpha_{\bar \tau}}:= \sup_{\tau \le \bar \tau} e^{-\tau \alpha} \| \Phi(\cdot, \tau)\|_{L^\infty_w}
\end{equation}
\begin{equation}\label{eq:Y}
    \| \psi \|_{Y^\beta_{ \bar t}} := \sup_{s \in (0, \bar t)} s^{-\beta} \| \psi \|_{L^r_t L^p_x(\Omega \times (0,s))} \, ,
\end{equation}
where $r,p \gg 1$ will be fixed later. The function spaces $X^{\alpha}_{\bar{\tau}}$ and $Y^\beta_{\bar{t}}$ consist of $C((-\infty,\bar{\tau}];L^\infty_w(\R^3))$ and measurable functions, respectively, with finite norm.
Let
\begin{equation}
    Z^{\alpha,\beta}_{\bar t} := X^\alpha_{\bar \tau}\times Y^\beta_{\bar t}
\end{equation}
endowed with the norm
\begin{equation}
    \| (\Phi, \psi) \|_{Z^{\alpha,\beta}_{\bar t}}
    = \| \Phi \|_{X^{\alpha}_{\bar \tau}} + \|\psi \|_{Y^\beta_{\bar t}}\, .
\end{equation}
We drop the dependence on $\bar \tau$ from $Z^{\alpha,\beta}_{\bar t}$ since we always assume that $\bar \tau = \log \bar t$.








We use the decomposition
\begin{equation}
	\label{eq:decomp}
	\Phi = \Phi^{\rm lin} + \Phi^{\rm per} \, ,
\end{equation}
where
\begin{equation}
    \Phi^{\rm lin}(\cdot, \tau) =  U^{\rm lin}(\cdot, \tau) = {\rm Re}(e^{\lambda \tau} \rho) 
\end{equation}
was defined in \eqref{eq: U lin}.

Our goal is to solve a set of integral equations for $\Phi^{\rm per}$ and $\psi$:
\begin{equation}\label{eq:integral eq}
    (\Phi^{\rm per},\psi) = L[(\Phi^{\rm per},\psi)] + B[(\Phi^{\rm per},\psi)] + G
\end{equation}
where $L = (L_i,L_o)$,  $B = (B_i,B_o)$, and $G = (G_i,G_o)$ will be specified below. The integral equations will be a reformulation of the inner and outer equations introduced in Section~\ref{sec:intro}.

We want to show that, for an appropriate choice of the parameters $\alpha$ and $\beta$, defined in~\eqref{eq:betadef}, and $r,p \gg 1$, there exists $\bar t>0$ such that the integral equations admit a unique solution $(\Phi^{\rm per},\psi) \in Z^{\alpha,\beta}_{\bar t}$. In what follows, we allow the implied constants to depend on $r$, $p$, and $a$.

We now determine the above operators, beginning with the inner integral equation.

\subsection{Inner integral equation}

Recall that the inner PDE is
\begin{equation}
    \label{eqn:innersecondround}
    \begin{split}
    \partial_\tau \Phi -\L_{\rm ss} &\Phi + \Phi \cdot \nabla (\Phi N) 
    + \div ( \tilde{N} \Phi \otimes \Psi  + \tilde{N} \Psi \otimes \Phi)
       \\ & \quad 
    + \bar{U} \cdot \nabla \Psi + \Psi \cdot \nabla \bar{U}+ \nabla \Pi =0\, ,
    \end{split}
\end{equation}
which must be satisfied on the support of $N$, and which we seek to solve in the whole space. With the decomposition~\eqref{eq:decomp},
we can derive an equation for $\Phi^{\rm per}$. The equation is
\begin{equation}\label{eq:Phi2}
    \partial_\tau  \Phi^{\rm per} - \L_{\rm ss}\Phi^{\rm per} = \bP \div \L[(\Phi^{\rm per},\psi)] + \bP \div \B[(\Phi^{\rm per},\psi)] + \bP \div \G \, ,
\end{equation}
where $\L$ is a linear operator in $(\Phi^{\rm per},\psi)$ given by
\begin{equation}\label{eq:Li}
\begin{split}
-\L[(\Phi^{\rm per},\psi)]  &=
 \underbrace{N\Phi^{\rm lin}  \otimes \Phi^{\rm per} + N\Phi^{\rm per}  \otimes \Phi^{\rm lin}}_{=: -\L_1[(\Phi^{\rm per},\psi)]} + \underbrace{\tilde{N} \Phi^{\rm lin} \otimes \Psi  + \tilde{N} \Psi \otimes \Phi^{\rm lin}}_{=: -\L_2[(\Phi^{\rm per},\psi)]}  \\
&\qquad + \underbrace{\bar{U} \otimes \Psi + \Psi \otimes \bar{U}}_{=: -\L_3[(\Phi^{\rm per},\psi)]} \, .
\end{split}
\end{equation}
The operator $\B[(\Phi^{\rm per},\psi)] = \B[(\Phi^{\rm per},\psi),(\Phi^{\rm per},\psi)]$ is induced by the bilinear form
\begin{equation}\label{eq:Bi}
-\B[(\Phi^{\rm per}_1,\psi_1),(\Phi^{\rm per}_2,\psi_2)]
 = 
\underbrace{N \Phi^{\rm per}_2  \otimes \Phi^{\rm per}_1}_{=: -\B_1[\Phi^{\rm per}_1,\Phi^{\rm per}_2]}
+ \underbrace{\tilde N \Phi^{\rm per}_1 \otimes \Psi_2 + \tilde{N} \Psi_2 \otimes  \Phi^{\rm per}_1}_{=: -\B_2[(\Phi^{\rm per}_1,\psi_1),(\Phi^{\rm per}_2,\psi_2)]}   \, .
\end{equation}
We finally have
\begin{equation}\label{eq:Gi}
-\G = N\Phi^{\rm lin} \otimes \Phi^{\rm lin}   \, .
\end{equation}
The associated integral operators are
\begin{equation}
\label{eq:lidef}
L_i[(\Phi^{\rm per},\psi)] = \int_{-\infty}^\tau e^{(\tau-s)\L_{\rm ss}} \bP \div \L[(\Phi^{\rm per},\psi)](\cdot,s) \, ds
\end{equation}
\begin{equation}
\label{eq:bidef}
B_i[(\Phi^{\rm per},\psi)] = \int_{-\infty}^\tau e^{(\tau-s)\L_{\rm ss}} \bP \div \B[(\Phi^{\rm per},\psi)](\cdot,s) \, ds
\end{equation}
\begin{equation}
\label{eq:gidef}
G_i = \int_{-\infty}^\tau e^{(\tau-s)\L_{\rm ss}} \bP \div \G(\cdot,s) \, ds \, .
\end{equation}

\subsection{Outer integral equation}

Let $\psi^{\rm div}[\Phi]$ be the solution of the Stokes equations with inhomogeneous divergence: 
When $\Omega$ is a smooth, bounded domain, we define $\psi^{\rm div}[\Phi]$ as in Lemma~\ref{lem:stokesdiv} with $h= - \nabla \eta \cdot \phi$. In the periodic setting, we set
\begin{equation}
    \psi^{\rm div}[\Phi] = - \nabla \Delta^{-1} (\nabla \eta \cdot \phi) \, ,
\end{equation}
see the discussion in Section~\ref{subsec:stokes periodic}.

Recall that the outer PDE is posed on $\Omega$ and reads
\begin{equation}
    \begin{cases}
     &\partial_t \psi  - \Delta \psi +  \div(\psi \otimes \psi) + (\psi \cdot \nabla \eta) \phi + (\phi \cdot \nabla \eta) \psi 
    \\  & \qquad
    - \phi \Delta \eta - 2 \nabla \phi \cdot \nabla  \eta  + (\phi \cdot \nabla \eta)\eta \phi +  \pi \nabla \eta + \nabla q  =0 
    \\&
    \div \psi = - \nabla \eta \cdot \phi \, .
    \end{cases}
\end{equation}
It will be convenient to rewrite, for each component $\phi_i$ of the vector field $\phi$,
\begin{equation}
    \nabla \phi_i \cdot \nabla \eta = \div (\phi_i \nabla \eta) - \phi_i \Delta \eta \, ,
\end{equation}
to keep everything in divergence form:
\begin{equation}
    \begin{cases}
    \label{eqn:outereqn}
    &\partial_t \psi  - \Delta \psi +  \div(\psi \otimes \psi) + (\psi \cdot \nabla \eta) \phi + (\phi \cdot \nabla \eta) \psi 
    \\  & \qquad
    + \phi \Delta \eta - 2 \div(\phi \otimes \nabla \eta) + (\phi \cdot \nabla \eta)\eta \phi + \pi \nabla \eta + \nabla q  =0 
    \\&
    \div \psi = - \nabla \eta \cdot \phi \, .
    \end{cases}
\end{equation}
The PDE is supplemented with the boundary condition $\psi|_{\p \Omega} = 0$. {The inner pressure $\pi$, which appears in the boundary term $\pi \nabla \eta$, is given by its similarity profile
\begin{equation}
\Pi = (-\Delta)^{-1} \div \div (\bar{U} \otimes \Phi^{\rm per} + \Phi^{\rm per} \otimes \bar{U} - \L[(\Phi^{\rm per},\psi)] - \B[(\Phi^{\rm per},\psi)] - \G ) \, .
\end{equation} 
Hence, we can rewrite it in physical variables as
\begin{equation}
	\pi = (-\Delta)^{-1} \div \div (\bar{u} \otimes \phi^{\rm per} + \phi^{\rm per} \otimes \bar{u} - {\bm \ell}[(\Phi^{\rm per},\psi)] - {\bm b}[(\Phi^{\rm per},\psi)] - {\bm g} ) \, ,
\end{equation}
where $\bm{\ell}$, $\bm{b}$, and $\bm{g}$ will represent $\bm{L}$, $\bm{B}$, and $\bm{G}$ in physical variables as opposed to similarity variables.
}

The integral equation for $\psi$ is
\begin{equation}
\begin{split}
    \psi = \psi^{\rm div}[\Phi] &-  \int_0^t e^{(t-s) A} \bP [\phi \Delta \eta  - 2\div (\phi \otimes \nabla \eta) + (\phi\cdot \nabla \eta)\eta \phi + \pi \nabla \eta ](\cdot,s) \, ds 
    \\& 
    - \int_0^t e^{(t-s) A} \bP [\div (\psi \otimes \psi) + (\psi \cdot \nabla \eta) \phi + (\phi \cdot \nabla \eta)\psi](\cdot,s) \, ds \, .
\end{split}
\end{equation}
We rewrite it as
\begin{equation}
    \label{eq:outerintegraleqn}
    \psi = L_o[(\Phi^{\rm per},\psi)] + B_o[(\Phi^{\rm per},\psi)] + G_o\, ,
\end{equation}
where $L_o$ acts linearly on $(\Phi^{\rm per},\psi)$ according to
\begin{equation}
    \label{eq:Lodef}
   \begin{split}
    &L_o[(\Phi^{\rm per},\psi)] =  \psi^{\rm div}[\Phi^{\rm per}] - \int_0^t e^{(t-s) A} \bP [\phi^{\rm per}  \Delta \eta  - 2 \div ( \phi^{\rm per} \otimes \nabla \eta)](\cdot,s) \, ds
    \\ &\quad  
     - \int_0^t e^{(t-s) A} \bP[ (\phi^{\rm lin}\cdot \nabla \eta)(\eta \phi^{\rm per} + \psi) +  ((\eta\phi^{\rm per} + \psi)\cdot \nabla \eta) \phi^{\rm lin}](\cdot,s) \, ds \\
     &\quad  - \int_0^t e^{(t-s) A} \bP [ (-\Delta)^{-1} \div \div ( \bar{u} \otimes \phi^{\rm per} + \phi^{\rm per} \otimes \bar{u} - \bm{\ell}[(\Phi^{\rm per},\psi)] )(\cdot,s) \nabla \eta] \, ds \, . 
\end{split}
\end{equation}
The operator $B_o$ is induced by the bilinear form
\begin{equation}
    \label{eq:Bodef}
\begin{split}
    &B_o[(\Phi^{\rm per}_1,\psi_1),(\Phi^{\rm per}_2,\psi_2)]  = - \int_0^t e^{(t-s) A} \bP [ \eta \phi_1^{\rm per}(\phi^{\rm per}_2\cdot \nabla \eta)  + \div (\psi_1 \otimes \psi_2) ](\cdot,s) \, ds 
    \\& \quad
    - \int_0^t e^{(t-s) A} \bP [ (\psi_1 \cdot \nabla \eta) \phi^{\rm per}_2  + (\phi_1^{\rm per}\cdot \nabla \eta)\psi_2 ](\cdot,s) \, ds \\
    &\quad  - \int_0^t e^{(t-s) A} \bP [ (-\Delta)^{-1} \div \div ( -\bm{b}[(\Phi^{\rm per}_1,\psi_1),(\Phi^{\rm per}_2,\psi_2)] )(\cdot,s) \nabla \eta] \, ds \, . 
\end{split}
\end{equation}
and finally,
\begin{equation}
    \label{eq:Godef}
    \begin{split}
    G_o = \psi^{\rm div}[\Phi^{\rm lin}] &- \int_0^t e^{(t-s) A} \bP [\phi^{\rm lin}  \Delta \eta  - 2 \div( \phi^{\rm lin}\otimes \nabla \eta) + \eta \phi^{\rm lin}(\phi^{\rm lin}\cdot \nabla \eta) ](\cdot,s) \, ds \\
    & - \int_0^t e^{(t-s) A} \bP [ (-\Delta)^{-1} \div \div ( -\bm{g}(\cdot,s) ) \nabla \eta] \, ds \, . 
    \end{split}
\end{equation}

\subsection{Elementary estimates}
We have the following elementary estimate for every $\psi \in Y^\beta_{\bar{t}}$. From now on, suppose that $\bar{t} \leq 1$. Let $\beta' \in (0,\beta)$. Then (extending $\psi$ by zero in time as necessary)
\begin{equation}
    \label{eq:tminusbetathing}
    \begin{split}
    \| t^{-\beta'} \psi \|_{L^r_t L^p_x(\Omega \times (0,\bar{t}))}  & \les \left( \sum_{k \leq 0} 2^{-k \beta' r} \| \mathbf{1}_{(2^{k-1},2^k)} \psi \|_{L^r_t L^p_x(\Omega \times (0,\bar{t}))}^r \right)^{\frac{1}{r}}
    \\& 
    \les \left( \sum_{k \leq 0} 2^{(\beta-\beta') k r} \right)^{\frac{1}{r}} \| \psi \|_{Y^\beta_{\bar{t}}} 
    \\&
    \les \| \psi \|_{Y^\beta_{\bar{t}}} \, ,
    \end{split}
\end{equation}
where the implied constants depend on $\beta,\beta',r$. Hence,
\begin{equation}
    \label{eq:weightoninside}
    \| e^{\tau (-\beta'-\frac{1}{2}+\frac{3}{2p}+\frac{1}{r})} \Psi \tilde N \|_{L^r_\tau L^p_\xi(\R^3 \times (-\infty,\bar{\tau}))} \les \| \psi \|_{Y^\beta_{\bar{t}}} \, ,
\end{equation}
where the $1/r$ arises from the change of measure $e^{\tau} \, d\tau = dt$.

Meanwhile, we have
\begin{equation}
    \| \phi(\cdot,t) \|_{L^\infty} = t^{-\frac{1}{2}} \| \Phi(\cdot,\tau) \|_{L^\infty} \, .
\end{equation}
Since $ \supp \nabla \eta \subset \{ \frac{1}{9} \leq |x| \leq \frac{1}{7} \}$, 
\begin{equation}\label{eq:boundary est}
    \| \phi(\cdot,t)  \|_{L^\infty(\supp \nabla \eta)} \les t^{\frac{3}{2}} \| \Phi(\cdot,\tau) \|_{L^\infty_w} \, .
\end{equation}

\section{Outer estimates}

We begin with the outer estimates. Crucially, we will see that the boundary terms from $\Phi^{\rm lin}$ will limit the decay rate $\beta$ of $\psi$.

Except in the $\psi^{\rm div}$ terms which correct the divergence, it will be convenient to work with pointwise estimates in time and use the observation that, for functions~$f$,
\begin{equation}
    \label{eq:holderforf}
    \| f \|_{L^r(0,t)} \leq t^{\frac{1}{r}} \| f \|_{L^\infty(0,t)} \, .
\end{equation}
Let
\begin{equation}
    \kappa := \kappa(r) = \frac{1}{r} + \frac{3}{2}  \, .
\end{equation}
This exponent will appear in the decay rates $\alpha$ and $\beta$. The $1/r$ term will be seen to come from~\eqref{eq:holderforf}. 
We recognize the exponent in~\eqref{eq:boundary est} as $\kappa-1/r$.

We then define
\begin{equation}
    \label{eq:betadef}
    \beta := \kappa + a - \frac{1}{8} \, , \quad \alpha := \kappa + a \, .
\end{equation}
Finally, we recall the estimate from Lemma~\ref{lem:efndecay},
\begin{equation}\label{eq:philin est}
     \| \Phi^{\rm lin}(\cdot,\tau) \|_{L^\infty_w} \les e^{\tau a} \, ,
     \quad \text{for any $\tau\in \R$}\, .
\end{equation}

\subsection{Estimate on $G_o$~\eqref{eq:Godef}}
We begin with the divergence term, which will already have the worst contribution. It is estimated using Lemma~\ref{lem:stokesdiv}, \eqref{eq:psidiv torus} (depending on whether $\Omega$ is a bounded or a periodic domain), \eqref{eq:boundary est} and \eqref{eq:philin est}:
\begin{equation}
    \label{eq:divterm1}
    \begin{aligned}
    \| \psi^{\rm div}[\Phi^{\rm lin}] \|_{L^r_t L^p_x(\Omega \times (0,t))} &\les \| \nabla \eta \cdot \phi^{\rm lin} \|_{L^r_t L^p_x(\Omega \times (0,t))} \\
    &\les t^{\kappa} \| \Phi^{\rm lin} \|_{L^\infty_\tau L^\infty_w(\R^3 \times (-\infty,\log t))} 
    \\&
    \les t^{\kappa+a} \, .
    \end{aligned}
\end{equation}
The remaining non-pressure terms are estimated using either Lemma~\ref{lem:stokessemigroup} or \eqref{eq:Stokes est torus}:
\begin{equation}
    \label{eq:gootherterms}
\begin{aligned}
    &\left\| \int_0^t e^{(t-s) A} \bP [\phi^{\rm lin} \Delta \eta  - 2 \div( \phi^{\rm lin}\otimes \nabla \eta) + \eta \phi^{\rm lin}(\phi^{\rm lin}\cdot \nabla \eta) ](\cdot,s) \, ds \right\|_{L^p(\Omega)} \\
    &\quad \les 
    \int_0^t 
    \| [\phi^{\rm lin} \Delta\eta] (\cdot,s)\|_{L^p(\Omega)}
    +
    (t-s)^{-\frac{1}{2}}\| \phi^{\rm lin}\otimes \nabla \eta (\cdot,s)\|_{L^p(\Omega)} \\
    &\quad\quad\quad + \| \phi^{\rm lin}(\cdot,s)\|_{L^\infty(\supp \nabla \eta)}^2
    \, ds
    \\
    &\quad \les \int_0^t  s^{a+\kappa-\frac{1}{r}} + (t-s)^{-\frac{1}{2}}  s^{a+\kappa-\frac{1}{r}} + s^{2(a + \kappa - \frac1r)} \, ds \\
    &\quad \les t^{a+\kappa-\frac{1}{r} + \frac{1}{2}} \, .
    \end{aligned}   
\end{equation}
{To estimate the pressure terms, we observe from~\eqref{eq:weightedlppressureest} that
\begin{equation}
\| (-\Delta)^{-1} \div \div (N \Phi^{\rm lin} \otimes \Phi^{\rm lin})(\cdot,s) \|_{L^p(\supp \nabla N)} \lesssim s^{2a + \frac{1}{2}(3-\frac{3}{p})} \, .
\end{equation}
Therefore, after changing variables,
\begin{equation}
	\label{eq:pressureestgo}
	\begin{aligned}
	&\left \| \int_0^t e^{(t-s) A} \bP [ (-\Delta)^{-1} \div \div (\eta \phi^{\rm lin} \otimes \phi^{\rm lin}) (\cdot,s) \nabla \eta ] \, ds \right \|_{L^p(\Omega)} \\
	&\quad \lesssim \int_0^t s^{2a-1+\kappa - \frac{1}{r}} \, ds \\
	&\quad \lesssim t^{2a+\kappa-\frac{1}{r}} \, .
	\end{aligned}
\end{equation} }

Notice that the $-1/r$ factors in~\eqref{eq:gootherterms} and~\eqref{eq:pressureestgo} will drop after applying~\eqref{eq:holderforf}. Combining~\eqref{eq:divterm1} and~\eqref{eq:gootherterms}, We conclude that for all $t \in (0,\bar{t})$,
\begin{equation}
    \label{eq:Goestfirst}
    \| G_o \|_{L^r_t L^p_x(\Omega \times (0,t))} \les t^{\kappa+a} \, . 
\end{equation}
Hence, \eqref{eq:betadef} gives
\begin{equation}
    \label{eq:Goestwithbeta}
    \| G_o \|_{Y^\beta_{\bar{t}}} 
    \les \bar{t}^{\kappa + a - \beta}
    = \bar{t}^{\frac{1}{8}} \, .
\end{equation}

\subsection{Estimate on $L_o$~\eqref{eq:Lodef}}

The terms in the first line of~\eqref{eq:Lodef} are estimated similarly to the $G_o$ estimate except that $a$ is replaced by $\alpha$ and $\Phi^{\rm lin}$ by $\Phi^{\rm per}$. 
We employ either Lemma~\ref{lem:stokessemigroup} or \eqref{eq:Stokes est torus} to estimate the remaining non-pressure terms:
\begin{equation}
    \begin{split}
       & \left\|  \int_0^t  e^{(t-s) A} \bP [ (\psi \cdot \nabla \eta) \phi^{\rm lin}](\cdot,s) \, ds \right\|_{L^p(\Omega)}
        \\ 
        &\quad  \lesssim
        \int_0^t  \| \psi(\cdot, s)\|_{L^p} \| \phi^{\rm lin}(\cdot, s) \|_{L^\infty(\supp \nabla \eta)}\, ds
        \\
        & \quad \lesssim 
        t^{\frac{3}{2} + a + 1 - \frac{1}{r} + \beta} \| \psi \|_{Y^\beta_{\bar t}}
        \\
        & \quad \lesssim
        t^{\kappa + a + \beta + 1}  \| \psi \|_{Y^\beta_{\bar t}}
        \, ,
    \end{split}
    \end{equation}
	and
	\begin{equation}
		\begin{split}
			&\left\| \int_0^t e^{(t-s) A} \bP[ (\phi^{\rm lin}\cdot \nabla \eta)(\eta \phi^{\rm per} + \psi) +  ((\eta\phi^{\rm per} + \psi)\cdot \nabla \eta) \phi^{\rm lin}](\cdot,s) \, ds \right\|_{L^p(\Omega)}
			\\&
			\les \int_0^t \| \phi^{\rm lin}(\cdot, s)|_A \|_{L^\infty(\Omega)}( \| \phi^{\rm per}(\cdot, s) \|_{L^p(\supp \nabla \eta)} + \| \psi(\cdot,s) \|_{L^p(\Omega)})\, ds
			\\&
			\les \int_0^t s^{\frac{3}{2} + a} (   s^{\frac{3}{2} + \alpha} \| \Phi^{\rm per}\|_{X^\alpha_{\bar \tau}}  + \| \psi(\cdot,s) \|_{L^p(\Omega)} )\, ds
			\\&
			\les t^{3+a+\alpha}
			\| \Phi^{\rm per}\|_{X^\alpha_{\bar \tau}} + t^{\frac{3}{2} + a + \beta + 1- \frac{1}{r}} \| \psi \|_{Y^\beta_{\bar t}}\, .
		\end{split}
	\end{equation}
{For the pressure terms, whenever $s \leq \bar{\tau}$, we have
\begin{equation}
	\label{eq:firstlinearpressureest}
	\begin{split}	
\| (-\Delta)^{-1} \div \div ( N \Phi^{\rm lin} \otimes \Phi^{\rm per}& + N \Phi^{\rm per} \otimes \Phi^{\rm lin} )(\cdot,s) \|_{L^p(\supp \nabla N)} 
\\ & 
\lesssim s^{a+\alpha+\frac{1}{2}(3-\frac{3}{p})} \| \Phi^{\rm per} \|_{X^\alpha_{\bar{\tau}}} \, ,
   \end{split}
\end{equation}
\begin{equation}
	\label{eq:onepointfivelinearpressureest}
	\begin{split}
	\| (-\Delta)^{-1} \div \div (\bar{U} \otimes \Phi^{\rm per} & + \Phi^{\rm per} \otimes \bar{U})(\cdot,s) \|_{L^p(\supp \nabla N)} \\
	&\lesssim s^{\alpha+\frac{1}{2}(3-\frac{3}{p})} \| \Phi^{\rm per} \|_{X^\alpha_{\bar{\tau}}} \, ,
	\end{split}
\end{equation}
\begin{equation}
	\label{eq:secondlinearpressureest}
	\begin{split}
		\| (-\Delta)^{-1} \div \div (\bar{U} \otimes \Psi + & \Psi \otimes \bar{U})(\cdot,s) \|_{L^p(\supp \nabla N)} 
		\\&
		\lesssim s^{\frac{1}{2}(3-\frac{3}{p})} \| \Psi(\cdot,s) \|_{L^p(\supp \bar{U})} \, ,
	\end{split}
\end{equation}
\begin{equation}
	\label{eq:thirdlinearpressureest}
	\begin{split}
		\| (-\Delta)^{-1} \div \div ( \tilde{N} \Phi^{\rm lin} \otimes \Psi +& \tilde{N} \Psi \otimes \Phi^{\rm lin})(\cdot,s)  \|_{L^p(\supp \nabla N)} 
		\\&
		\lesssim s^{a+\frac{1}{2}(3-\frac{3}{p})} \| \tilde{N} \Psi(\cdot,s) \|_{L^p} \, .
	\end{split}
\end{equation}
The terms~\eqref{eq:firstlinearpressureest} and~\eqref{eq:onepointfivelinearpressureest} lead to an estimate similar to the pressure term in $G_o$ but with powers $\bar{t}^{a+\alpha+\kappa}$ and $\bar{t}^{\alpha+\kappa}$ when measured in $L^r_t L^p_x(\Omega \times (0,\bar{t}))$. For the term~\eqref{eq:secondlinearpressureest}, we have
\begin{equation}
\begin{aligned}
&\left\| \int_0^t e^{(t-s)A} \bP [ (-\Delta)^{-1} \div \div (\bar{u} \otimes \psi + \psi \otimes \bar{u})(\cdot,s) \nabla \eta ] \, ds \right\|_{L^p(\Omega)} \\
&\quad \lesssim \int_0^t s^{-1+\kappa-\frac{1}{r}} \| \psi(\cdot,s) \|_{L^p(\Omega)} \, ds \\
&\quad \lesssim t^{\kappa-\frac{2}{r}+\beta} \| \psi \|_{Y^\beta_{\bar{t}}} \, .
\end{aligned}
\end{equation}
The contribution of the term~\eqref{eq:thirdlinearpressureest} is similar but with exponent $t^{a+\kappa-\frac{2}{r}+\beta}$.
}

We conclude that
\begin{equation}
    \label{eq:Loestwithbeta}
    \| L_o[(\Phi^{\rm per},\psi)] \|_{Y^\beta_{\bar{t}}} 
    \les \bar{t}^{\kappa + \alpha - \beta} \| \Phi^{\rm per} \|_{X^\alpha_{\bar{\tau}}} 
     + \bar t ^{\kappa -\frac{2}{r} } \| \psi\|_{Y_{\bar t}^\beta} 
    \lesssim
    \bar t^{\frac{1}{2}} \| (\Phi^{\rm per}, \psi) \|_{Z^{\alpha,\beta}_{\bar t}} \, . 
\end{equation}

\subsection{Estimate on $B_o$~\eqref{eq:Bodef}}
By the semigroup estimates in Lemma~\ref{lem:stokessemigroup} (or \eqref{eq:Stokes est torus}, in the periodic setting), for all $t \in (0,\bar{t})$, we have
\begin{equation}
\label{eq:Bo1term}
\begin{aligned}
    &\left\| \int_0^t e^{(t-s)A} \bP \div [\psi_1 \otimes \psi_2](\cdot,s) \, ds \right\|_{L^p(\Omega)} \\
    &\quad \les \int_0^t (t-s)^{-\frac{1}{2}-\frac{3}{2p}} \| \psi_1(\cdot,s)\|_{L^p} \| \psi_2 (\cdot,s)\|_{L^p}  \, ds \\
    &\quad\les \left( \int_0^t  (t-s)^{(-\frac{1}{2}-\frac{3}{2p})(2r)'} \, ds \right)^{\frac{1}{(2r)'}} \| \psi_1 \|_{L^r_t L^p_x(\Omega \times (0,t))}  \| \psi_2 \|_{L^r_t L^p_x(\Omega \times (0,t))} \\
    &\quad \les t^{2\beta} \| \psi_1 \|_{Y^\beta_{\bar{t}}} \| \psi_2 \|_{Y^\beta_{\bar{t}}}  \, ,
    \end{aligned}
\end{equation}
where we choose $p,r \gg 1$ such that the first term is time integrable.
Moreover,
\begin{equation}
 \begin{split}
 	& \left|   \int_0^t e^{(t-s)A}\mathbb{P}[\eta \phi_1^{\rm per}(\phi_2^{\rm per}\cdot \nabla \eta)(\cdot,s)]\, ds
 	\right|_{L^p(\Omega)}
 	\\& 
 	\les \int_0^t \| \phi_1^{\rm per}(\cdot,s) \|_{L^p(\supp \nabla \eta)}  \| \phi_2^{\rm per}(\cdot,s) \|_{L^\infty(\supp \nabla \eta)}\, ds
 	\\&
 	\les t^{4 + 2\alpha} \| \Phi_1^{\rm per}\|_{X^\alpha_{\bar \tau}} \| \Phi_2^{\rm per}\|_{X^\alpha_{\bar \tau}}\, ,
 \end{split}
\end{equation}
and 
\begin{equation}
\label{eq:Bo2term}
\begin{aligned}
    &\left\| \int_0^t e^{(t-s)A} \bP [(\psi_1 \cdot \nabla \eta)\phi_2^{\rm per} + (\phi_1^{\rm per}\cdot \nabla \eta)\psi_2 ](\cdot,s) \, ds \right\|_{L^p(\Omega)} \\
    & \les \int_0^t  s^{\kappa - \frac{1}{r}} (\| \psi_1 \|_{L^p} \| \Phi_2^{\rm per} \|_{L^\infty_w} + \| \psi_2 \|_{L^p} \| \Phi_1^{\rm per} \|_{L^\infty_w} ) \, ds \\
    & \les \left( \int_0^t  s^{(\kappa -\frac{1}{r} + \alpha)r'}  \, ds \right)^{\frac{1}{r'}} 
    ( \| \psi_1 \|_{L^r_t L^p_x(\Omega \times (0,t))}  \| \Phi_2^{\rm per} \|_{X^\alpha_{\bar{\tau}}} 
    + \| \psi_2 \|_{L^r_t L^p_x(\Omega \times (0,t))}  \| \Phi_1^{\rm per} \|_{X^\alpha_{\bar{\tau}}}
     ) \\
    & \les t^{\kappa + \alpha+\beta}
    ( \| \psi_1 \|_{Y^\beta_{\bar{t}}} \| \Phi_2^{\rm per} \|_{X^\alpha_{\bar{\tau}}} + \| \psi_2 \|_{Y^\beta_{\bar{t}}} \| \Phi_1^{\rm per} \|_{X^\alpha_{\bar{\tau}}} )  \, .
    \end{aligned}
\end{equation}
{Finally, the pressure terms are estimated similarly to the $G_o$ term and the term~\eqref{eq:thirdlinearpressureest} except with $\Phi^{\rm per}$ replacing $\Phi^{\rm lin}$.}

Combining~\eqref{eq:Bo1term} and~\eqref{eq:Bo2term} with~\eqref{eq:holderforf} (also, $\alpha \geq \beta$), we have
\begin{equation}
\label{eq:Boestwithbeta}
    \| B_o[(\Phi^{\rm per}_1,\psi_1),(\Phi^{\rm per}_2,\psi_2)] \|_{Y^{\beta}_{\bar{t}}} \les \bar{t}^{\beta} \| (\Phi^{\rm per}_1,\psi_1) \|_{Z^{\alpha,\beta}_{\bar{t}}} \| (\Phi^{\rm per}_2,\psi_2) \|_{Z^{\alpha,\beta}_{\bar{t}}} \, .
\end{equation}

\section{Inner estimates}
We now turn to the inner estimates, for which our main tool is Lemma~\ref{lemma:spacialdecay}.

\subsection{Estimate on $G_i$~\eqref{eq:gidef},~\eqref{eq:Gi}}
For all $\tau \in (-\infty,\bar{\tau})$, we have (with $\delta = a/2$, in Lemma~\ref{lemma:spacialdecay}),
\begin{equation}
\begin{aligned}
    \| G_i(\cdot,\tau) \|_{L^\infty_w} &= \left\| \int_{-\infty}^\tau e^{(\tau-s) \L_{\rm ss}} \bP \div \G(\cdot,s) \, ds \right\|_{L^\infty_w} \\
    &\les \int_{-\infty}^\tau e^{(\tau-s)(\frac{3a}{2})} (\tau-s)^{-\frac{1}{2}} e^{2as} \, ds \les e^{2a\tau} \, ,
    \end{aligned}
\end{equation}
that is,
\begin{equation}
\label{eq:Giestwithalpha}
    \| G_i \|_{X^\alpha_{\bar{\tau}}} \les e^{(2a-\alpha) \bar{\tau}} \, .
\end{equation}
Notice that $2a - \alpha = a - \kappa \ge 1$ provided $r\gg 1$ and $a\ge 5$.

\subsection{Estimate on $B_i$~\eqref{eq:bidef},~\eqref{eq:Bi}}
The estimate for the $\B_1$ terms is analogous to the $G_i$ estimate. For all $\tau \in (-\infty,\bar{\tau})$, we have
\begin{equation}
   \left\| \int_{-\infty}^\tau e^{(\tau-s) \L_{\rm ss}} \bP \div \B_1[\Phi^{\rm per}_1,\Phi^{\rm per}_2](\cdot,s) \, ds \right\|_{L^\infty_w} \les e^{2\alpha \tau} \| \Phi_1^{\rm per} \|_{X^{\alpha}_{\bar{\tau}}} \| \Phi_2^{\rm per} \|_{X^{\alpha}_{\bar{\tau}}} \, .
\end{equation}
For the $\B_2$ terms, we apply Lemma \ref{lemma:spacialdecay} and  \eqref{eq:weightoninside} to get 
\begin{equation}
\begin{aligned}
    &\left\| \int_{-\infty}^\tau e^{(\tau-s) \L_{\rm ss}} \bP \div \B_2[(\Phi^{\rm per}_1,\psi_1),(\Phi^{\rm per}_2,\psi_2)](\cdot,s) \, ds \right\|_{L^\infty_w} \\
    &\quad \les_\delta \int_{-\infty}^{\tau} e^{(\tau-s)(a+\delta)} (\tau-s)^{-\frac{1}{2}-\frac{3}{2p}}\| \Phi_1^{\rm per}(\cdot,s)  \|_{L^\infty_w}  \| \Psi_2\tilde N(\cdot,s)  \|_{L^p} \, ds \\
    &\quad 
    \les_\delta \int_{-\infty}^\tau e^{(\tau-s)(a+\delta)} (\tau-s)^{-\frac{1}{2}-\frac{3}{2p}} e^{s\alpha} e^{(\beta'+\frac{1}{2}-\frac{3}{2p}-\frac{1}{r})s}  \\
    &\quad\quad\quad\quad\quad\quad \times \| e^{(-\beta'-\frac{1}{2}+\frac{3}{2p}+\frac{1}{r})s} \Psi_2 \tilde N (\cdot,s) \|_{L^p} \, ds  \;  \| \Phi_1^{\rm per} \|_{X^\alpha_{\bar{\tau}}}  \\
    &\quad \overset{\eqref{eq:weightoninside}}{\les_{\delta,\beta'}} e^{\tau(\alpha+\beta'+\frac{1}{2}-\frac{3}{2p}-\frac{1}{r})}  \| \Phi_1^{\rm per} \|_{X^\alpha_{\bar{\tau}}}  \| \Psi_2 \|_{Y^\beta_{\bar{t}}} \\
    &\quad \les_{\delta,\beta'} e^{(\alpha+\beta) \tau} \| \Phi_1^{\rm per} \|_{X^\alpha_{\bar{\tau}}}  \| \Psi_2 \|_{Y^\beta_{\bar{t}}} 
    \end{aligned}
\end{equation}
where $\beta'+1/2-3/(2p)-1/r=\beta$ and $\delta = \beta/2$. Combining the above two estimates, we conclude
\begin{equation}
    \label{eq:Biestwithalpha}
    \| B_i[(\Phi^{\rm per}_1,\psi_1),(\Phi^{\rm per}_2,\psi_2)] \|_{X^\alpha_{\bar{\tau}}} \les e^{\beta \bar{\tau}} \| (\Phi^{\rm per}_1,\psi_1) \|_{Z^{\alpha,\beta}_{\bar{t}}} \| (\Phi^{\rm per}_2,\psi_2) \|_{Z^{\alpha,\beta}_{\bar{t}}} \, .
\end{equation}

\subsection{Estimate on $L_i$~\eqref{eq:lidef},~\eqref{eq:Li}}

The estimate for the $\L_1$ terms is analogous to the $G_i$ and $\B_1$ estimates. For all $\tau \in (0,\bar{\tau})$, we have
\begin{equation}
    \left\| \int_{-\infty}^\tau e^{(\tau-s) \L_{\rm ss}} \bP \div \L_1[(\Phi^{\rm per},\psi)](\cdot,s) \, ds \right\|_{L^\infty_w} \les e^{\tau(a+\alpha)} \| \Phi^{\rm per} \|_{X^\alpha_{\bar{\tau}}} \, .
\end{equation}
The estimates for the $\L_2$ terms is analogous to the $\B_2$ estimate:
\begin{equation}
    \left\| \int_{-\infty}^\tau e^{(\tau-s) \L_{\rm ss}} \bP \div \L_2[(\Phi^{\rm per},\psi)](\cdot,s) \, ds \right\|_{L^\infty_w} \les e^{(a+\beta) \tau}  \| \Psi_2 \|_{Y^\beta_{\bar{t}}} \, .
\end{equation}
Finally, we have
\begin{equation}
\begin{aligned}
    &\left\| \int_{-\infty}^\tau e^{(\tau-s) \L_{\rm ss}} \bP \div \L_3[(\Phi^{\rm per},\psi)](\cdot,s) \, ds \right\|_{L^\infty_w} \\
    &\quad\les_\delta \int_{-\infty}^\tau e^{(\tau-s) (a+\delta)} (\tau-s)^{-\frac{1}{2}-\frac{3}{2p}} \| \bar{U} \|_{L^\infty_w} \| \Psi \|_{L^p} \, ds \\
    &\quad \les_{\delta,\beta'} e^{(\beta'+\frac{1}{2}-\frac{3}{2p}-\frac{1}{r})\tau } \| \psi \|_{Y^\beta_{\bar{t}}}\\
    &\quad\les_{\delta,\beta'} e^{(\beta+\frac{1}{4})\tau} \| \psi \|_{Y^\beta_{\bar{t}}} \, , 
    \end{aligned}
\end{equation}
where $\beta' + 1/4 - 3/(2p) - 1/r = \beta$ and $\delta = (a+\beta)/2 - a$. Combining the above three estimates and $a\geq 10$, we have
\begin{equation}
\label{eq:Liestwithalpha}
    \| L_i[(\Phi^{\rm per},\psi)] \|_{X^\alpha_{\bar{\tau}}} \les e^{(\beta+\frac{1}{4}-\alpha) \bar{\tau}}\| (\Phi^{\rm per},\psi) \|_{Z^{\alpha,\beta}_{\bar{t}}} \overset{\eqref{eq:betadef}}{\les} e^{\frac{1}{8} \bar{\tau}}\| (\Phi^{\rm per},\psi) \|_{Z^{\alpha,\beta}_{\bar{t}}}  \, .
\end{equation}

\section{Conclusion}
\label{sec:conclusion}

We now collect the estimates \eqref{eq:Goestwithbeta}, \eqref{eq:Loestwithbeta}, \eqref{eq:Boestwithbeta}, \eqref{eq:Giestwithalpha}, \eqref{eq:Biestwithalpha}, \eqref{eq:Liestwithalpha}, which yield that
\begin{equation}
    \| L \|_{Z^{\alpha,\beta}_{\bar t} \to Z^{\alpha,\beta}_{\bar t}} 
    + \| B \|_{Z^{\alpha,\beta}_{\bar t} \times Z^{\alpha,\beta}_{\bar t} \to Z^{\alpha,\beta}_{\bar t}} 
    + \| G \|_{Z^{\alpha,\beta}_{\bar t}} \to 0 \text{ as } \bar{t} \to 0^+ \, ,
\end{equation}
with the appropriate choices of $\alpha$ and $\beta$ in~\eqref{eq:betadef}, $p,r \gg 1$, and $a \geq 10$. In particular, there exists $\bar{t} \ll 1$ such that
\begin{equation}
    L + B + G : 
    \{\| (\Phi^{\rm per}, \psi)\|_{Z^{\alpha,\beta}_{\bar t}}\le 1\}
    \to 
    \{ \| (\Phi^{\rm per}, \psi)\|_{Z^{\alpha,\beta}_{\bar t}}\le 1\}
\end{equation}
is a contraction, cf. \cite[Subsection 4.2.2]{albritton2021non}.  Hence, there exists a unique solution $(\Phi^{\rm per},\psi)$ to the integral equation~\eqref{eq:integral eq} in the above ball. By the ansatz~\eqref{eq:ansatz} and decomposition~\eqref{eq:decomp}, the solution $(\Phi^{\rm per},\psi)$ determines a mild Navier-Stokes solution $u \: \Omega \times (0,\bar{t}) \to \R^3$ with forcing $f$ and satisfying
\begin{equation}
    u \in L^r_t L^p_x(\Omega \times (\varepsilon,\bar{t})) \, ,
\end{equation}
for all $\varepsilon \in (0,\bar{t})$. 

That the solution is indeed mild is a technical point, which we now justify.
Initially, we know that, for all divergence-free $w \in C^1_c((0,T);C^2\cap C_0(\Omega))$, we have
\begin{equation}
    \label{eq:theequationieneed}
    \int_0^{\bar{t}} \int_\Omega u(-\p_t w - \Delta w) \, dx \, dt = \int_0^{\bar{t}} \int_\Omega u \otimes u : \nabla w + f \cdot w \, dx \,dt \, ,
\end{equation}
and $u(\cdot,t) \in L^p_\sigma(\Omega)$ for a.e. $t \in (0,\bar{t})$. In particular, $u = \bP u$, and it is weakly continuous in $(0,\bar{t})$ due to~\eqref{eq:theequationieneed}. Consider $\varepsilon \in (0,\bar{t})$ such that $u(\cdot,\varepsilon) \in L^p_\sigma(\Omega)$. Let $v$ be the mild solution to the Stokes equations on $\Omega \times (\varepsilon,\bar{t})$ with initial data $v(\cdot,\varepsilon) = u(\cdot,\varepsilon)$ and right-hand side $- \div u \otimes u + f$. Then $u - v$ is a very weak solution in the sense of Lemma~\ref{lem:stokesdiv} with zero initial data, zero right-hand side, and zero divergence. By uniqueness, $u \equiv v$ on $\Omega \times (\varepsilon,\bar{t})$.

We begin by justifying that $u \neq \bar{u}$, which is necessary for non-uniqueness. Recall that $\| \Phi^{\rm lin}N(\cdot,\tau)\|_{L^p} \ges e^{\tau a}$ and $\| \Phi^{\rm per}(\cdot,\tau) \|_{L^\infty_w} \les e^{\tau \alpha}$ for all sufficiently negative $\tau$. Additionally, due to~\eqref{eq:weightoninside}, we have that, for all $\beta' < \beta$, $\| \Psi(\cdot,\tau_k) \|_{L^p} \les e^{\tau_k(\beta'-\frac{1}{2}+\frac{3}{2p}-\frac{1}{r})}$ along a sequence $\tau_k \to -\infty$; in particular, the exponent on the right-hand side can be made strictly greater than $a$. Hence, $\| \Phi N(\cdot,\tau_k) + \Psi(\cdot,\tau_k)\|_{L^p} \ges e^{\tau_k a}$ for large enough $k$, which justifies the claim.

We now justify that the above solution is a Leray-Hopf solution with right-hand side. Since $L^r_t L^p_x(\Omega \times (\varepsilon,\bar{t}))$ is a subcritical space when $2/r + 3/p < 1$ and $f$ is smooth away from $t=0$, it is classical that $u \in L^\infty_t (W^{1,q}_0)_x(\Omega \times (\varepsilon,\bar{t}))$ for all $q \in (1,+\infty)$ (bootstrap using the mild formulation and the linear estimates in Lemma~\ref{lem:stokessemigroup}) and, moreover, satisfies energy equality on $\Omega \times (\varepsilon,\bar{t})$, for all $\varepsilon \in (0,\bar{t})$ (see~\cite[Theorem 1.4.1, p. 272]{Sohrbook}, for example). 
It remains to show that $\| u(\cdot,t_k) \|_{L^2} \to 0$ as $k \to +\infty$ for some sequence of times $t_k \to 0^+$. We have $\| \bar{u}(\cdot,t) \|_{L^2} + \| \phi \eta(\cdot,t) \|_{L^2} \les t^{1/4}$,
 and $\| \psi(\cdot,t_k) \|_{L^2} \to 0$ follows from~\eqref{eq:tminusbetathing}. This completes the proof of Theorem~\ref{thm:introthm}.

\subsubsection*{Acknowledgments} DA was supported by NSF Postdoctoral Fellowship  Grant No.\ 2002023 and Simons Foundation Grant No.\ 816048.
EB was supported by Giorgio and Elena Petronio Fellowship.
MC was supported by the SNSF Grant 182565.

\section{Appendix}

\begin{lemma}[A convolution inequality]
    \label{lem:convolutioninequality}
Let $d \in \N$, $\alpha, \beta \in (d,+\infty)$ and $\delta \in (0,1]$. Then
\begin{equation}
   I_{d,\alpha,\beta}(\delta) := \int_{\R^d} \la x-y \ra^{-\alpha} \left\la \frac{y}{\delta} \right\ra^{-\beta} \, dy \les_{d,\alpha,\beta} \la x \ra^{-\min(\alpha,\beta)} \delta^d \, .
\end{equation}
\end{lemma}
\begin{proof}
We will suppress dependence on $d,\alpha,\beta$ when convenient.

For $|x| \leq 1$, we have
\begin{equation}
    I \les \int_{\R^d} \left\la \frac{y}{\delta} \right\ra^{-\beta} \, dy \les \delta^d \, ,
\end{equation}
so we restrict our attention to $|x| \geq 1$.

In the region $R_1 := \{ |y| \leq |x|/2 \}$, we have $|x-y| \approx |x|$ and
\begin{equation}
\int_{R_1} \la x-y \ra^{-\alpha} \left\la \frac{y}{\delta} \right\ra^{-\beta} \, dy \les \la x \ra^{-\alpha} \int_{|y| \leq |x|/2}  \left\la \frac{y}{\delta} \right\ra^{-\beta} \, dy  \les \la x \ra^{-\alpha} \delta^d \, .
\end{equation}

In the region $R_2 := \{ |x-y| \leq |x|/2 \}$, we have $|y| \approx |x|$ and
\begin{equation}
\int_{R_2} \la x-y \ra^{-\alpha} \left\la \frac{y}{\delta} \right\ra^{-\beta} \, dy  \les \left\la \frac{x}{\delta} \right\ra^{-\beta} \int_{|x-y| \leq |x|/2} \la x-y \ra^{-\alpha} \, dy  \les \left\la x \right\ra^{-\beta} \delta^\beta \, ,
\end{equation}
where  $|x| \geq 1$ and $\delta \in (0,1]$ ensure that $\la x/\delta \ra \approx \la x \ra \delta$.

In the region $R_3 := \R^d \setminus (R_1 \cup R_2)$, we have $|y| \approx |x-y|$ and
\begin{equation}
\int_{R_3} \la x-y \ra^{-\alpha} \left\la \frac{y}{\delta} \right\ra^{-\beta} \, dy  \les \int_{r \geq |x|/2} r^{-\alpha} r^{-\beta} \delta^{\beta} r^{d-1} \, dr \les \la x \ra^{-\alpha-\beta+d} \delta^\beta \, ,
\end{equation}
where we again use that $|x| \geq 1$ and $\delta \in (0,1]$ to make simplifications.

Finally, we sum the above estimates to complete the proof when $|x| \geq 1$.
\end{proof}

\begin{remark}
    \label{rmk:convolutionremark}
As a consequence, we have the following variant, which is useful in the proof of Lemma~\ref{lemma:spacialdecay}. Let $\zeta, \beta > d$, $p \in [1,+\infty]$, and $p'$ be its H{\"o}lder conjugate. Then, for all $f \in L^p_\zeta$, we have
\begin{equation}
\begin{aligned}
    \int_{\R^d} |f(x-y)| \left\la \frac{y}{\delta} \right\ra^{-\beta} &\leq \| f \|_{L^p_\zeta} \times [ I_{d,\zeta p',\beta p'}(\delta) ]^{\frac{1}{p'}} \\
    &\les_{d,\alpha,\beta,p} \| f \|_{L^p_\zeta} \la x \ra^{-\min(\zeta,\beta)} \delta^{\frac{d}{p'}} \, .
    \end{aligned}
\end{equation}
with obvious adjustments when $p = +\infty$.
\end{remark}

\bibliographystyle{abbrv}
\bibliography{nonuniquenessbib}

\end{document}